\documentclass[11pt]{article}
\usepackage{amsmath,amssymb,amsthm}
\usepackage[margin=1.1in]{geometry}
\usepackage{hyperref}

\newtheorem{theorem}{Theorem}[section]
\newtheorem{lemma}[theorem]{Lemma}
\newtheorem{proposition}[theorem]{Proposition}
\newtheorem{corollary}[theorem]{Corollary}
\theoremstyle{remark}
\newtheorem{remark}[theorem]{Remark}

\newcommand{\Res}[1]{\operatorname*{Res}_{#1}}
\newcommand{\re}{\operatorname{Re}}
\newcommand{\im}{\operatorname{Im}}
\numberwithin{equation}{section}

\title{Proof of a Conjectured Ramanujan Type Master Theorem for Powers of the Cosecant}
\author{Zachary P. Bradshaw}
\date{QodeX Quantum, Inc.\\ zak@qodexquantum.ai}

\begin{document}
\maketitle

\begin{abstract}
We prove a conjecture of Bradshaw and Atale asserting a Ramanujan type master theorem for the kernel $\pi^m/\sin^m(\pi s)$, whose poles at the non-positive integers have order $m$. The residue data at these poles are organized by a family of polynomials studied by Airault with coefficients the central factorial numbers of the first kind. The proof reduces the conjecture to a Laurent series identity for $\csc^m$, established by induction from an elementary differential recursion, and the resulting theorem holds under a growth hypothesis strictly weaker than Hardy's. As applications, we obtain integral representations for powers of the cosecant, a closed form for the Mellin convolution powers of the Cauchy kernel $(1+x)^{-1}$, and a family of differential identities satisfied by the Airault polynomials.
\end{abstract}

\section{Introduction}\label{sec:intro}

Ramanujan's master theorem~\cite{amdeberhan2012ramanujan} states that if $f$ admits an expansion
\begin{equation}\label{eq:rmt-series}
f(x)=\sum_{n=0}^{\infty}\frac{(-1)^n}{n!} g(n) x^n
\end{equation}
in which the coefficient sequence $g(n)$ extends naturally to a function of a complex variable, then the Mellin transform of $f$ is given by
\begin{equation}\label{eq:rmt}
\int_0^\infty x^{s-1}f(x) dx=\Gamma(s) g(-s).
\end{equation}
Ramanujan's own derivation, communicated to Hardy in his quarterly reports~\cite{berndt1983quarterly,berndt2012ramanujan}, was formal; Hardy~\cite{hardy1999ramanujan} supplied a rigorous proof under natural hypotheses. Let $H(\delta)=\{z\in\mathbb{C}:\re(z)\ge-\delta\}$. Hardy's theorem asserts that if $g$ is analytic on $H(\delta)$ for some $0<\delta<1$ and satisfies the growth condition
\begin{equation}\label{eq:hardy-growth}
|g(v+iw)|<Ce^{Pv+A|w|}
\end{equation}
for $v+iw\in H(\delta)$ and constants $C,P\in\mathbb{R}$ and $A<\pi$, then
\begin{equation}\label{eq:hardy-rmt}
\int_0^\infty x^{s-1}\sum_{k=0}^{\infty}(-1)^k g(k) x^k dx=\frac{\pi}{\sin(\pi s)} g(-s),
\qquad 0<\re(s)<\delta.
\end{equation}
This result has been generalized in various ways~\cite{atale2021certain,bradshaw2025generalized,bradshaw2023operational,ding1996ramanujan,olafsson2012ramanujan,olafsson2013ramanujan}, perhaps most notably a multivariate version known as the method of brackets~\cite{bradshaw2023compatibility,bradshaw2023explorations,gonzalez2020extension,gonzalez2017extension,gonzalez2010method,gonzalez2010definite}, which has been used to evaluate Feynman diagrams~\cite{ananthanarayan2023method,gonzalez2010methodfeynman}.

The proof rests on two pillars~\cite{bradshaw2023explorations,flajolet1995mellin}. The first is a residue computation, in which the poles of $\pi/\sin(\pi z)$ at the non-positive integers generate the power series in the integrand. The second is the Mellin inversion theorem. Nothing in this architecture is specific to $\pi/\sin(\pi s)$~\cite{bradshaw2025generalized}. Replacing it by a meromorphic function $h(s)$ with poles at the non-positive integers produces a family of Ramanujan-like master theorems of the shape
\begin{equation}\label{eq:meromorphic-mt}
\int_0^\infty x^{s-1}\sum_{k=0}^{\infty}\Res{z=-k}\bigl(h(z) g(-z) x^{-z}\bigr) dx=h(s) g(-s),
\end{equation}
proved rigorously when the poles of $h$ are simple~\cite[Theorem~2]{bradshaw2025generalized} and heuristically for poles of arbitrary order~\cite[Theorem~5]{bradshaw2025generalized}, where the residues acquire factors of $\log(x)$~\cite{bleistein1986asymptotic,paris2001asymptotics}. The natural test case for the higher-order theory is the $m$-th power $\pi^m/\sin^m(\pi s)$, whose poles at the non-positive integers have exact order $m$. Numerical experimentation in~\cite{bradshaw2025generalized} led to the following conjecture. Define a sequence of polynomials by
\begin{equation}\label{eq:airault-rec}
P_1(x)=1,\qquad P_2(x)=x,\qquad P_m(x)=\bigl(x^2+(m-2)^2\pi^2\bigr)P_{m-2}(x),\quad m>2.
\end{equation}
These polynomials appear (up to rescaling) in work of Airault~\cite{airault2008hyperbolic} on Fourier transforms of hyperbolic measures, and by writing $P_{m}(\pi t)=\pi^{m-1}\widetilde{P}_{m}(t)$, the coefficients of $\widetilde{P}_{m}$ become central factorial numbers of the first kind~\cite{butzer1989central}. Conjecture~1 of~\cite{bradshaw2025generalized} asserts that, under suitable growth conditions on $g$,
\begin{equation}\label{eq:conjecture}
\int_0^\infty x^{s-1}\sum_{n=0}^{\infty}(-1)^{mn}\Bigl[P_m\Bigl(\tfrac{d}{dz}+\log(x)\Bigr)g(z)\Bigr]_{z=n}x^n dx
=\frac{(-1)^{m-1}(m-1)! \pi^m}{\sin^m(\pi s)} g(-s).
\end{equation}
The purpose of this paper is to prove this conjecture, with an explicit and, as it turns out, \emph{weaker} growth hypothesis than Hardy's: the constant $\pi$ in \eqref{eq:hardy-growth} may be replaced by $m\pi$.

\begin{theorem}[Cosecant-power master theorem]\label{thm:A}
Let $m\ge 1$ be an integer and let $g$ be analytic on $H(\delta)$ for some $0<\delta<1$, satisfying
\begin{equation}\label{eq:growth-A}
|g(v+iw)|\le Ce^{Pv+A|w|}
\end{equation}
for $v+iw\in H(\delta)$ and constants $C,P\in\mathbb{R}$ and $A<m\pi$. Then the series in \eqref{eq:conjecture} converges absolutely for $0<x<e^{-P}$, extends to $(0,\infty)$ by the inverse Mellin integral, and the identity \eqref{eq:conjecture} holds for all $0<\re(s)<\delta$.
\end{theorem}

The proof reduces the conjecture to a single Laurent-series identity. The principal part of $\pi^m\csc^m(\pi w)$ at $w=0$ is generated by the coefficients of $P_m$ (Lemma~\ref{lem:principal-part}). This identity is in turn a consequence of the elementary differential recursion
\begin{equation}
\frac{d^2}{du^2}\csc^m(u)=m(m+1)\csc^{m+2}(u)-m^2\csc^m(u),
\end{equation}
which under the residue calculus translates \emph{exactly} into the Airault recursion \eqref{eq:airault-rec}. The periodicity of the sine then transports the computation from $w=0$ to every pole (Lemma~\ref{lem:residue}), and Hardy's contour argument, with the arc estimate strengthened by the $m$-fold exponential decay of $\csc^m$ off the real axis, completes the proof.

Already the case $g=1$ carries nontrivial content. For odd $m$, it produces the evaluation
\begin{equation}
\int_0^\infty x^{s-1} \frac{P_m(\log(x))}{1+x} dx=\frac{(m-1)! \pi^m}{\sin^m(\pi s)},
\qquad 0<\re(s)<1,
\end{equation}
and for even $m$, the corresponding evaluation with $1-x$ in the denominator, in which the singularity of the integrand at $x=1$ is removable because $P_m$ vanishes at the origin. These evaluations identify the Mellin convolution powers of the Cauchy kernel in closed form. Writing $K_1(x)=(1+x)^{-1}$ and $K_m=K_{m-1}\star K_1$, where $\star$ denotes multiplicative convolution, we prove
\begin{equation}
K_m(x)=\frac{(-1)^{m-1}}{(m-1)!}\cdot\frac{P_m(\log(x))}{1-(-1)^mx},
\end{equation}
a formula which we have not found in the literature, although the underlying density is classical (see Remark~\ref{rem:probabilistic}). The same circle of ideas yields the differential identities
\begin{equation}
P_m\Bigl(\frac{d}{ds}\Bigr)\Bigl[\frac{\pi}{\sin(\pi s)}\Bigr]=\frac{(-1)^{m-1}(m-1)! \pi^m}{\sin^m(\pi s)}
\end{equation}
for odd $m$, together with even counterparts for the cotangent, and these verify Theorem~\ref{thm:A} independently when $g=1$. A further specialization promotes equations (5.24) to (5.26) of~\cite{bradshaw2025generalized} from experimental to proven.

The paper is organized as follows. Section~\ref{sec:prelim} collects notation, the basic properties of the polynomials $P_m$, and the contour estimates. Section~\ref{sec:proof} proves Theorem~\ref{thm:A}. Section~\ref{sec:applications} develops the special cases and applications.

\section{Preliminaries}\label{sec:prelim}

\subsection{Notation and the Hardy contour}

Throughout, $x>0$ is real and $x^{-z}=e^{-z\log(x)}$ with the real logarithm. We write $z=v+iw$ with $v,w\in\mathbb{R}$, and $H(\delta)=\{z:\re(z)\ge-\delta\}$. We abbreviate the commuting operators
\begin{equation}
D=\frac{d}{dz},\qquad L=\log(x),
\end{equation}
so that for $g$ analytic near $z=n$, 
\begin{equation}
    [P(D+L)g(z)]_{z=n}=\sum_j p_j\sum_{k=0}^{j}\binom{j}{k}(\log(x))^{j-k}g^{(k)}(n)
\end{equation}
for a polynomial $P(x)=\sum_j p_j x^j$.

Fix $0<c<\delta<1$. The \emph{Hardy contour} of index $N\in\mathbb{N}$ is the closed curve consisting of the vertical segment from $c-iR_N$ to $c+iR_N$, traversed upward, together with the left semicircle
\begin{equation}
\gamma_N=\{z=c+R_Ne^{i\theta}:\ \theta\in[\tfrac{\pi}{2},\tfrac{3\pi}{2}]\},
\qquad
R_N=N+\tfrac12+c,
\end{equation}
traversed so that the region $\{|z-c|<R_N,\ \re(z)<c\}$ is enclosed with positive orientation. The radius is chosen so that $\gamma_N$ meets the real axis at $c-R_N=-N-\tfrac12$, a half-integer. The poles of $\csc^m(\pi z)$ enclosed by the contour of index $N$ are therefore exactly $z=0,-1,\dots,-N$.

\subsection{The Airault polynomials}

The polynomials $P_m$ are defined by the recursion \eqref{eq:airault-rec}. The first cases are $P_3(x)=x^2+\pi^2$, $P_4(x)=x^3+4\pi^2x$,
$P_5(x)=x^4+10\pi^2x^2+9\pi^4$, and $P_6(x)=x^5+20\pi^2x^3+64\pi^4x$.
Let us write
\begin{equation}
P_m(x)=\sum_{j=0}^{m-1}p_{m,j} x^j .
\end{equation}

\begin{proposition}\label{prop:airault}
For every $m\ge1$ the following hold.
\begin{enumerate}
\item[(i)] $P_m$ is monic of degree $m-1$; in particular $p_{m,m-1}=1$.
\item[(ii)] $P_m(-x)=(-1)^{m-1}P_m(x)$; equivalently $p_{m,j}=0$ unless $j\equiv m-1\pmod 2$.
\item[(iii)] $P_m$ factors over $\mathbb{C}$ as
\begin{equation}
P_m(x)=x^{\varepsilon_m}\prod_{\substack{1\le j\le m-2\\ j\equiv m\ (\mathrm{mod}\ 2)}}\bigl(x^2+j^2\pi^2\bigr),
\qquad \varepsilon_m=\begin{cases}0,& m\ \text{odd},\\ 1,& m\ \text{even},\end{cases}
\end{equation}
so that the zeros of $P_m$ are simple and located at $x=ij\pi$ for $j\equiv m\pmod 2$ and $|j|\le m-2$, including $j=0$ when $m$ is even.
\item[(iv)] The coefficients satisfy, for $m\ge1$ and all $j\in\mathbb{Z}$, with the convention $p_{m,j}=0$ for $j<0$ or $j>m-1$,
\begin{equation}\label{eq:coeff-rec}
p_{m+2,j}=p_{m,j-2}+m^2\pi^2 p_{m,j}.
\end{equation}
\end{enumerate}
\end{proposition}

\begin{proof}
All four statements are immediate inductions on the recursion $P_{m+2}(x)=(x^2+m^2\pi^2)P_m(x)$, anchored at $P_1=1$ and $P_2=x$. For (iii), the factor appended at step $m\to m+2$ is $x^2+m^2\pi^2$, and $m\equiv m+2\pmod2$; the base cases carry $\varepsilon_1=0$ and $\varepsilon_2=1$. Statement (iv) is the recursion read off coefficientwise.
\end{proof}

\subsection{The cosecant estimate on the Hardy contour}
 
\begin{lemma}\label{lem:csc-bound}
There is a constant $K$ such that for all $0<c<1$, all $N\ge1$, and all $z=v+iw$ on the arc $\gamma_N$,
\begin{equation}\label{eq:csc-bound}
|\csc(\pi z)|\le K e^{-\pi|w|}.
\end{equation}
\end{lemma}
 
\begin{proof}
From $\sin(\pi z)=\sin(\pi v)\cos(i\pi w)+\cos(\pi v)\sin(i\pi w)$, one has the standard identity
\begin{equation}\label{eq:sin-modulus}
|\sin(\pi z)|^2=\sin^2(\pi v)+\sinh^2(\pi w).
\end{equation}
\emph{Case 1, $|w|\ge\tfrac12$.} Then $|\sin(\pi z)|\ge\sinh(\pi|w|)=\tfrac12(e^{\pi|w|}-e^{-\pi|w|})\ge\tfrac12(1-e^{-\pi})e^{\pi|w|}$, so
\begin{equation}
|\csc(\pi z)|\le\frac{2}{1-e^{-\pi}} e^{-\pi|w|}.
\end{equation}
\emph{Case 2, $|w|\le\tfrac12$.} Write $z=c+R_Ne^{i\theta}$ with $\theta\in[\pi/2,3\pi/2]$, so $w=R_N\sin(\theta)$ and $v=c+R_N\cos(\theta)$ with $\cos(\theta)\le0$. From $|w|\le\tfrac12$ we get $|\sin(\theta)|\le 1/(2R_N)$, hence
\begin{equation}
0\le 1+\cos(\theta) = 1-\sqrt{1-\sin^2(\theta)}\le\sin^2(\theta)\le\frac{1}{4R_N^2},
\end{equation}
using $1-\sqrt{1-t}\le t$ for $t\in[0,1]$. Therefore
\begin{equation}
-N-\tfrac12=c-R_N \le v=c-R_N+R_N(1+\cos(\theta)) \le c-R_N+\frac{1}{4R_N} \le -N-\tfrac12+\frac{1}{4R_N},
\end{equation}
and since $R_N\ge\tfrac32$, the distance from $v$ to the nearest integer is at least $\tfrac12-\tfrac16=\tfrac13>\tfrac14$. Consequently $|\sin(\pi v)|\ge\sin(\pi/4)=1/\sqrt2$, and by \eqref{eq:sin-modulus},
\begin{equation}
|\csc(\pi z)|\le\sqrt2\le\sqrt2 e^{\pi/2}e^{-\pi|w|},
\end{equation}
since $e^{-\pi|w|}\ge e^{-\pi/2}$ in this case. Both cases give \eqref{eq:csc-bound} with $K=\max\bigl\{2/(1-e^{-\pi}),\ \sqrt{2} e^{\pi/2}\bigr\}$, which is independent of $c$, $N$, and $z$.
\end{proof}
 
\begin{remark}\label{rem:vertical}
The two-case argument also yields a bound on vertical segments, with a constant now depending on $c$. On the segment $\re(z)=c$ with $0<c<1$ we will use
\begin{equation}
|\csc(\pi z)|\le \max\Bigl\{e^{\pi/2}\csc(\pi c), \frac{2}{1-e^{-\pi}}\Bigr\} e^{-\pi|w|},
\end{equation}
which follows from \eqref{eq:sin-modulus} by the same split. For $|w|\le\tfrac12$ one has $|\sin(\pi z)|\ge|\sin(\pi c)|$ and $e^{-\pi|w|}\ge e^{-\pi/2}$, while Case~1 applies verbatim for $|w|\ge\tfrac12$.
\end{remark}

\section{Proof of the theorem}\label{sec:proof}

Throughout this section, we fix an integer $m\ge1$ and write
\begin{equation}\label{eq:Cm-def}
C_m(w)=\frac{\pi^m}{\sin^m(\pi w)},
\qquad
h_m(z)=(-1)^{m-1}(m-1)! C_m(z),
\end{equation}
so that $h_1(z)=\pi/\sin(\pi z)$ is Hardy's kernel and $h_m$ is the kernel on the right side of \eqref{eq:conjecture}. The function $C_m$ is meromorphic on $\mathbb{C}$ with poles of exact order $m$ at each integer and no other singularities. We denote by
\begin{equation}\label{eq:laurent-def}
C_m(w)=\sum_{r=-m}^{\infty}c^{(m)}_r w^r,\qquad 0<|w|<1,
\end{equation}
its Laurent expansion at $w=0$, and we extend the notation by setting $c^{(m)}_r=0$ for $r<-m$.

\subsection{The principal part of $C_m$}

\begin{lemma}[Parity]\label{lem:parity}
$C_m(-w)=(-1)^mC_m(w)$; consequently $c^{(m)}_r=0$ unless $r\equiv m\pmod 2$.
\end{lemma}

\begin{proof}
Immediate from the oddness of the sine. The coefficient statement follows by comparing the expansions of $C_m(-w)$ and $(-1)^mC_m(w)$.
\end{proof}

\begin{lemma}[Differential recursion]\label{lem:ode}
For every $m\ge1$,
\begin{equation}\label{eq:ode}
C_{m+2}(w)=\frac{C_m''(w)+m^2\pi^2 C_m(w)}{m(m+1)} .
\end{equation}
\end{lemma}

\begin{proof}
Let $u=\pi w$ and $f_m(u)=\csc^m(u)$. Since $\frac{d}{du}\csc(u)=-\csc(u)\cot(u)$,
\begin{equation}
f_m'(u)=-m\csc^m(u) \cot(u),
\end{equation}
and differentiating once more, using $\frac{d}{du}\cot(u)=-\csc^2(u)$,
\begin{equation}
f_m''(u)=-m\Bigl[(-m\csc^m(u) \cot(u))\cot(u)+\csc^m(u) (-\csc^2(u))\Bigr]
=m^2\csc^m(u) \cot^2(u)+m\csc^{m+2}(u) .
\end{equation}
Substituting $\cot^2(u)=\csc^2(u)-1$,
\begin{equation}\label{eq:csc-ode}
f_m''(u)=m(m+1)\csc^{m+2}(u)-m^2\csc^m(u) .
\end{equation}
Since $C_m(w)=\pi^mf_m(\pi w)$, we have $C_m''(w)=\pi^{m+2}f_m''(\pi w)$, and \eqref{eq:csc-ode} gives
\begin{equation}
C_m''(w)=\pi^{m+2}\bigl[m(m+1)\csc^{m+2}(\pi w)-m^2\csc^m(\pi w)\bigr]
=m(m+1) C_{m+2}(w)-m^2\pi^2 C_m(w).
\end{equation}
Solving for $C_{m+2}$ proves \eqref{eq:ode}.
\end{proof}

\begin{lemma}[Coefficient recursion]\label{lem:coeff-laurent}
For every $m\ge1$ and every $j\ge0$,
\begin{equation}\label{eq:laurent-rec}
c^{(m+2)}_{-j-1}=\frac{j(j-1) c^{(m)}_{1-j}+m^2\pi^2 c^{(m)}_{-j-1}}{m(m+1)} .
\end{equation}
\end{lemma}

\begin{proof}
Laurent series may be differentiated termwise on their annulus of convergence, so
\begin{equation}
C_m''(w)=\sum_{r\ge-m}r(r-1) c^{(m)}_rw^{r-2}=\sum_{r\ge-m-2}(r+2)(r+1) c^{(m)}_{r+2} w^{r}.
\end{equation}
Comparing coefficients of $w^r$ in \eqref{eq:ode} yields $m(m+1) c^{(m+2)}_r=(r+2)(r+1) c^{(m)}_{r+2}+m^2\pi^2 c^{(m)}_r$ for all $r$. Setting $r=-j-1$, so that $(r+2)(r+1)=(1-j)(-j)=j(j-1)$ and $c^{(m)}_{r+2}=c^{(m)}_{1-j}$, gives \eqref{eq:laurent-rec}.
\end{proof}

\begin{lemma}[Principal part of $C_m$]\label{lem:principal-part}
For every $m\ge1$ and every $j\ge0$,
\begin{equation}\label{eq:principal-part}
c^{(m)}_{-j-1}=(-1)^{m-1+j} \frac{j!}{(m-1)!} p_{m,j},
\end{equation}
with the convention $p_{m,j}=0$ for $j\ge m$ from Proposition~\ref{prop:airault}. Equivalently, the principal part of $C_m$ at $w=0$ is
\begin{equation}
\frac{(-1)^{m-1}}{(m-1)!}\sum_{j=0}^{m-1}(-1)^j j! p_{m,j} w^{-j-1}.
\end{equation}
\end{lemma}

\begin{proof}
We induct on $m$ in steps of two; the base cases $m=1$ and $m=2$ anchor the odd and even chains respectively.

\noindent\emph{Base case $m=1$.} From $\sin(\pi w)=\pi w\bigl(1-\tfrac{(\pi w)^2}{6}+O(w^4)\bigr)$ we get
\begin{equation}
C_1(w)=\frac{\pi}{\sin(\pi w)}=\frac1w\Bigl(1+\frac{(\pi w)^2}{6}+O(w^4)\Bigr)=w^{-1}+\frac{\pi^2}{6} w+O(w^3),
\end{equation}
where in the second equality, $1/(1-(\pi w)^2/6+O(w^4))$ was expanded as a geometric series. It follows that $c^{(1)}_{-1}=1$ and $c^{(1)}_{-j-1}=0$ for $j\ge1$. The right side of \eqref{eq:principal-part} at $m=1$ is $(-1)^{j} j! p_{1,j}$, which equals $1$ for $j=0$ since $P_1=1$, and $0$ for $j\ge1$ since $\deg P_1=0$. 

\noindent\emph{Base case $m=2$.} Squaring the expansion of $C_1$,
\begin{equation}
C_2(w)=C_1(w)^2=\Bigl(w^{-1}+\frac{\pi^2}{6}w+O(w^3)\Bigr)^2=w^{-2}+\frac{\pi^2}{3}+O(w^2),
\end{equation}
so $c^{(2)}_{-2}=1$, $c^{(2)}_{-1}=0$, and $c^{(2)}_{-j-1}=0$ for $j\ge2$. The right side of \eqref{eq:principal-part} at $m=2$ is $(-1)^{1+j}j! p_{2,j}$, which equals $1$ at $j=1$ as $p_{2,1}=1$, equals $0$ at $j=0$ as $p_{2,0}=0$, and equals $0$ for $j\ge2$. Both base cases hold.

\noindent\emph{Induction step.} Assume \eqref{eq:principal-part} holds at $m$ for all $j\ge0$. Fix $j\ge0$ and insert the inductive hypothesis into \eqref{eq:laurent-rec}. If $j\ge2$,
\begin{align*}
c^{(m+2)}_{-j-1}
&=\frac{1}{m(m+1)}\Bigl[j(j-1) (-1)^{m-1+(j-2)}\frac{(j-2)!}{(m-1)!} p_{m,j-2}
+m^2\pi^2 (-1)^{m-1+j}\frac{j!}{(m-1)!} p_{m,j}\Bigr]\\
&=\frac{(-1)^{m-1+j}}{(m-1)! m(m+1)}\Bigl[j(j-1)(j-2)! p_{m,j-2}+m^2\pi^2 j! p_{m,j}\Bigr]\\
&=(-1)^{m-1+j} \frac{j!}{(m+1)!}\Bigl[p_{m,j-2}+m^2\pi^2 p_{m,j}\Bigr]\\
&=(-1)^{(m+2)-1+j} \frac{j!}{\bigl((m+2)-1\bigr)!}\;p_{m+2,j},
\end{align*}
where the last line uses the coefficient recursion \eqref{eq:coeff-rec} of Proposition~\ref{prop:airault}(iv) and $(-1)^{m+1+j}=(-1)^{m-1+j}$. If $j\in\{0,1\}$, the factor $j(j-1)$ vanishes, and the same computation goes through with the first bracketed term absent, matching $p_{m+2,j}=m^2\pi^2p_{m,j}$, which is \eqref{eq:coeff-rec} with $p_{m,j-2}=0$. This establishes \eqref{eq:principal-part} at $m+2$ for all $j\ge0$ and completes the induction.

We note two consistency checks. For $j\ge m+2$, both $p_{m,j-2}$ and $p_{m,j}$ vanish, and the computation returns $c^{(m+2)}_{-j-1}=0$, consistent with the pole order of $C_{m+2}$. For $j=m$ and $j=m+1$, only the first bracketed term survives, and it propagates the leading and subleading data; in particular $c^{(m)}_{-m}=p_{m,m-1}=1$ for all $m$, as it must since $C_m(w)\sim w^{-m}$. Finally, \eqref{eq:principal-part} respects Lemma~\ref{lem:parity}, since $c^{(m)}_{-j-1}\ne0$ forces $-j-1\equiv m\pmod2$, that is $j\equiv m-1\pmod2$, which is precisely the parity constraint on $p_{m,j}$ from Proposition~\ref{prop:airault}(ii).
\end{proof}

\subsection{The residue lemma}

\begin{lemma}[Taylor expansion of the shifted test function]\label{lem:taylor}
Let $g$ be analytic on $H(\delta)$, let $n\ge0$ be an integer, and let $x>0$. Define
\begin{equation}
G_n(w)=g(n-w) e^{-w\log(x)} .
\end{equation}
Then $G_n$ is analytic on $|w|<n+\delta$, with Taylor expansion
\begin{equation}\label{eq:Gn-taylor}
G_n(w)=\sum_{j=0}^{\infty}\frac{(-1)^j}{j!}\Bigl[(D+L)^jg(z)\Bigr]_{z=n} w^j,
\end{equation}
where $D=\frac{d}{dz}$ and $L=\log(x)$.
\end{lemma}

\begin{proof}
Analyticity on $|w|<n+\delta$ is clear, since $\re(n-w)>-\delta$ there. Multiply the two absolutely convergent expansions
\begin{equation}
g(n-w)=\sum_{k\ge0}\frac{g^{(k)}(n)}{k!}(-w)^k,
\qquad
e^{-wL}=\sum_{i\ge0}\frac{(-L)^i}{i!} w^i .
\end{equation}
The Cauchy product gives $G_n(w)=\sum_{j\ge0}a_jw^j$ with
\begin{equation}
a_j=\sum_{k+i=j}\frac{(-1)^kg^{(k)}(n)}{k!}\cdot\frac{(-1)^iL^i}{i!}
=\frac{(-1)^j}{j!}\sum_{k=0}^{j}\binom{j}{k}L^{ j-k}g^{(k)}(n)
=\frac{(-1)^j}{j!}\Bigl[(D+L)^jg(z)\Bigr]_{z=n},
\end{equation}
the final equality being the binomial theorem for the commuting operators $D$ and multiplication by the scalar $L$.
\end{proof}

\begin{lemma}[Residue lemma]\label{lem:residue}
Let $g$ be analytic on $H(\delta)$, let $x>0$, and let $n\ge0$ be an integer. Then
\begin{equation}\label{eq:residue-lemma}
\Res{z=-n}\Bigl(h_m(z) g(-z) x^{-z}\Bigr)
=(-1)^{mn}\Bigl[P_m\Bigl(\tfrac{d}{dz}+\log(x)\Bigr)g(z)\Bigr]_{z=n} x^n .
\end{equation}
\end{lemma}

\begin{proof}
Substitute $z=w-n$; residues are invariant under translation, so
\begin{equation}
\Res{z=-n}\bigl(C_m(z)g(-z)x^{-z}\bigr)=\Res{w=0}\bigl(C_m(w-n) g(n-w) x^{n-w}\bigr).
\end{equation}
By the addition formula, $\sin(\pi(w-n))=\sin(\pi w)\cos(\pi n)-\cos(\pi w)\sin(\pi n)=(-1)^n\sin(\pi w)$, so that
\begin{equation}
C_m(w-n)=\frac{\pi^m}{(-1)^{mn}\sin^m(\pi w)}=(-1)^{mn} C_m(w),
\end{equation}
and $g(n-w)x^{n-w}=x^n G_n(w)$ in the notation of Lemma~\ref{lem:taylor}. Therefore
\begin{equation}\label{eq:res-shift}
\Res{z=-n}\bigl(C_m(z)g(-z)x^{-z}\bigr)=(-1)^{mn}x^n \Res{w=0}\bigl(C_m(w) G_n(w)\bigr).
\end{equation}
On a punctured disk $0<|w|<\min(1,n+\delta)$ both expansions \eqref{eq:laurent-def} and \eqref{eq:Gn-taylor} converge, and the product $C_mG_n$ has Laurent expansion obtained by termwise multiplication. Write $C_m=\Pi_m+\Theta_m$, where $\Pi_m(w)=\sum_{j=0}^{m-1}c^{(m)}_{-j-1}w^{-j-1}$ is the principal part and $\Theta_m$ is analytic at $0$. The product $\Theta_mG_n$ is analytic at $0$ and contributes nothing to the residue; the residue of $\Pi_mG_n$ is the coefficient of $w^{-1}$, namely
\begin{equation}
\Res{w=0}\bigl(C_mG_n\bigr)=\sum_{j=0}^{m-1}c^{(m)}_{-j-1} a_j
=\sum_{j=0}^{m-1}(-1)^{m-1+j}\frac{j!}{(m-1)!} p_{m,j}\cdot\frac{(-1)^j}{j!}\Bigl[(D+L)^jg\Bigr]_{z=n},
\end{equation}
by Lemmas~\ref{lem:principal-part} and \ref{lem:taylor}. The signs combine as $(-1)^{m-1+j}(-1)^j=(-1)^{m-1}$ and the factorials cancel, so
\begin{equation}
\Res{w=0}\bigl(C_mG_n\bigr)=\frac{(-1)^{m-1}}{(m-1)!}\sum_{j=0}^{m-1}p_{m,j}\Bigl[(D+L)^jg\Bigr]_{z=n}
=\frac{(-1)^{m-1}}{(m-1)!}\Bigl[P_m(D+L) g\Bigr]_{z=n}.
\end{equation}
Multiplying \eqref{eq:res-shift} by $(-1)^{m-1}(m-1)!$ and recalling $h_m=(-1)^{m-1}(m-1)! C_m$ gives \eqref{eq:residue-lemma}.
\end{proof}

\subsection{Proof of Theorem \ref{thm:A}}

We first record the two quantitative lemmas, namely the convergence of the residue series and the vanishing of the arc term.

\begin{lemma}[Convergence of the residue series]\label{lem:convergence}
Let $g$ satisfy the hypotheses of Theorem~\ref{thm:A} and set $\rho=\delta/2$. Then for every $x>0$ and every integer $n\ge0$,
\begin{equation}\label{eq:coef-bound}
\Bigl|\bigl[P_m(D+L)g\bigr]_{z=n}\Bigr|
\le C_2(x) e^{Pn},
\end{equation}
where
\begin{equation}
C_2(x)=C_1\sum_{j=0}^{m-1}|p_{m,j}| j! \Bigl(|\log(x)|+\tfrac{2}{\delta}\Bigr)^{j}
\end{equation}
and $C_1=Ce^{|P|\rho+A\rho}$. Consequently, the series
\begin{equation}\label{eq:fm-series}
f_m(x)=\sum_{n=0}^{\infty}(-1)^{mn}\bigl[P_m(D+L)g\bigr]_{z=n} x^n
\end{equation}
converges absolutely, locally uniformly on $0<x<e^{-P}$.
\end{lemma}

\begin{proof}
For $n\ge0$ the closed disk $\overline{D}(n,\rho)$ lies in $H(\delta)$, since $\re(z)\ge n-\rho\ge-\delta/2>-\delta$ there. Cauchy's estimates on this disk give, using \eqref{eq:growth-A},
\begin{equation}
|g^{(k)}(n)|\le \frac{k!}{\rho^{k}}\max_{|z-n|=\rho}|g(z)|
\le \frac{k!}{\rho^{k}} C e^{Pn+|P|\rho+A\rho}
= \frac{k!}{\rho^{k}} C_1 e^{Pn},
\end{equation}
since $\re(z)\in[n-\rho,n+\rho]$ on the circle gives $e^{P\re(z)}\le e^{Pn}e^{|P|\rho}$ for either sign of $P$, and $|\im(z)|\le\rho$. Hence for $0\le j\le m-1$,
\begin{equation}
\Bigl|\bigl[(D+L)^jg\bigr]_{z=n}\Bigr|
\le\sum_{k=0}^{j}\binom{j}{k}|L|^{ j-k} |g^{(k)}(n)|
\le C_1e^{Pn}\sum_{k=0}^{j}\binom{j}{k}|L|^{ j-k}\frac{k!}{\rho^{k}}
\le C_1e^{Pn} j! \Bigl(|L|+\frac1\rho\Bigr)^{j},
\end{equation}
where the last step uses $k!\le j!$ and the binomial theorem. Summing against $|p_{m,j}|$ gives \eqref{eq:coef-bound}. Since $C_2$ is bounded on compact subsets of $(0,\infty)$ and $\limsup_{n\to\infty}\bigl(C_2(x)e^{Pn}x^n\bigr)^{1/n}=xe^{P}<1$ for $x<e^{-P}$, the series \eqref{eq:fm-series} converges absolutely and locally uniformly there by the root test.
\end{proof}

\begin{lemma}[Arc estimate]\label{lem:arc}
Let $g$ satisfy the hypotheses of Theorem~\ref{thm:A}, fix $0<c<\delta$ and $0<x<e^{-P}$, and set $\mu=-\log(xe^P)$ and $\nu=\min(m\pi-A,\ \mu)$, both positive. Then
\begin{equation}
\Bigl|\int_{\gamma_N}h_m(z) g(-z) x^{-z} dz\Bigr|
\le C_3 R_N e^{-\nu R_N}\xrightarrow[N\to\infty]{}0,
\end{equation}
with $C_3$ independent of $N$.
\end{lemma}

\begin{proof}
Let $z=v+iw=c+R_Ne^{i\theta}\in\gamma_N$ with $\theta\in[\pi/2,3\pi/2]$, so that $v=c+R_N\cos(\theta)$ with $\cos(\theta)\le0$ and $w=R_N\sin(\theta)$. Since $v\le c<\delta$, the point $-z$ satisfies $\re(-z)=-v\ge-c>-\delta$, so $-z\in H(\delta)$ and \eqref{eq:growth-A} applies, giving
\begin{equation}
|g(-z)|\le Ce^{-Pv+A|w|}.
\end{equation}
Moreover $|x^{-z}|=x^{-v}$ and, by Lemma~\ref{lem:csc-bound}, $|C_m(z)|\le\pi^mK^me^{-m\pi|w|}$. Hence
\begin{equation}
\bigl|h_m(z)g(-z)x^{-z}\bigr|
\le(m-1)! \pi^mK^mC\;e^{(A-m\pi)|w|}\;e^{-v(P+\log(x))} .
\end{equation}
Now $-v(P+\log(x))=v\mu$ and $v=c+R_N\cos(\theta)$, so $e^{-v(P+\log(x))}=e^{c\mu}e^{\mu R_N\cos(\theta)}$; also $|w|=R_N|\sin(\theta)|$. On $[\pi/2,3\pi/2]$ we have $\cos(\theta)=-|\cos(\theta)|$, so the exponent obeys
\begin{equation}
\begin{aligned}
(A-m\pi)R_N|\sin(\theta)|+\mu R_N\cos(\theta)
&=-R_N\Bigl[(m\pi-A)|\sin(\theta)|+\mu|\cos(\theta)|\Bigr]\\
&\le-R_N \nu \bigl(|\sin(\theta)|+|\cos(\theta)|\bigr)\le-\nu R_N,
\end{aligned}
\end{equation}
using $|\sin(\theta)|+|\cos(\theta)|\ge\sin^2(\theta)+\cos^2(\theta)=1$. Since $|dz|=R_N d\theta$,
\begin{equation}
\Bigl|\int_{\gamma_N}h_m(z) g(-z) x^{-z} dz\Bigr|\le(m-1)! \pi^mK^mC e^{c\mu}\;R_N\int_{\pi/2}^{3\pi/2}e^{-\nu R_N} d\theta
=C_3R_Ne^{-\nu R_N}\to0 .\qedhere
\end{equation}
\end{proof}

\begin{proof}[Proof of Theorem~\ref{thm:A}]
Fix $0<c<\delta$ and first let $0<x<e^{-P}$. Consider the function
\begin{equation}
F(z)=h_m(z) g(-z) x^{-z},
\end{equation}
which is meromorphic on the closed region bounded by the Hardy contour of index $N$. Indeed, $x^{-z}$ is entire, $g(-z)$ is analytic there because the region lies in $\{\re(z)\le c\}$ so that $\re(-z)\ge-c>-\delta$, and $h_m$ is meromorphic with poles of order $m$ precisely at the integers. The integers enclosed by the contour of index $N$ are those in the interval $(-N-\tfrac12, c)$, namely $z=0,-1,\dots,-N$, and no pole lies on the contour, whose real crossings are $c\in(0,1)$ and $-N-\tfrac12$. The contour, consisting of the vertical segment traversed upward and the arc $\gamma_N$ traversed from $c+iR_N$ through $c-R_N$ to $c-iR_N$, is positively oriented, so the residue theorem gives
\begin{equation}\label{eq:residue-thm}
\frac{1}{2\pi i}\int_{c-iR_N}^{c+iR_N}F(z) dz
+\frac{1}{2\pi i}\int_{\gamma_N}F(z) dz
=\sum_{n=0}^{N}\Res{z=-n}F(z)
=\sum_{n=0}^{N}(-1)^{mn}\bigl[P_m(D+L)g\bigr]_{z=n}x^n,
\end{equation}
the last equality by Lemma~\ref{lem:residue}. We now let $N\to\infty$. The arc term vanishes by Lemma~\ref{lem:arc}, and the right side converges to $f_m(x)$ by Lemma~\ref{lem:convergence}. For the vertical term, we need integrability on the whole line $\re(z)=c$. For $z=\sigma+it$ with $\sigma\in[a,b]\subset(0,\delta)$, the two-case argument of Lemma~\ref{lem:csc-bound}, with $|\sin(\pi z)|\ge|\sin(\pi\sigma)|\ge\sin(\pi\min(a,1-b))$ when $|t|\le\tfrac12$ and with $|\sin(\pi z)|\ge\tfrac12(1-e^{-\pi})e^{\pi|t|}$ when $|t|\ge\tfrac12$, yields
\begin{equation}\label{eq:vertical-bound}
|\csc(\pi(\sigma+it))|\le K_{a,b} e^{-\pi|t|},
\end{equation}
where
\begin{equation}
K_{a,b}=\max\Bigl\{\frac{e^{\pi/2}}{\sin(\pi\min(a,1-b))},\ \frac{2}{1-e^{-\pi}}\Bigr\},
\end{equation}
uniformly in $\sigma\in[a,b]$, since $\sin(\pi\sigma)\ge\sin(\pi\min(a,1-b))$ for $\sigma\in[a,b]\subset(0,1)$ by the concavity of the sine on $[0,\pi]$. Hence
\begin{equation}\label{eq:vertical-integrand}
|F(\sigma+it)|\le(m-1)! \pi^mK_{a,b}^m C e^{-P\sigma}x^{-\sigma} e^{(A-m\pi)|t|},
\end{equation}
which is integrable in $t$ since $A<m\pi$. Passing to the limit in \eqref{eq:residue-thm},
\begin{equation}\label{eq:inverse-mellin}
\frac{1}{2\pi i}\int_{c-i\infty}^{c+i\infty}h_m(s) g(-s) x^{-s} ds=f_m(x),
\qquad 0<x<e^{-P}.
\end{equation}

Now define $f_m(x)$ for \emph{all} $x>0$ by the left side of \eqref{eq:inverse-mellin}. By \eqref{eq:vertical-integrand} and Cauchy's theorem applied to rectangles with vertical sides $\re(s)=c_1$ and $\re(s)=c_2$, whose horizontal sides tend to zero by \eqref{eq:vertical-integrand}, the definition is independent of $c\in(0,\delta)$, and it agrees with the series \eqref{eq:fm-series} on $(0,e^{-P})$.

Finally, we apply the Mellin inversion theorem in the form used in~\cite[Theorem~2]{bradshaw2025generalized}; see also Chapter~8 of~\cite{debnath2016integral}, as well as \cite{titchmarsh1948introduction}. Suppose $\varphi$ is analytic on the strip $a<\re(s)<b$ and $\int_{-\infty}^{\infty}|\varphi(\sigma+it)| dt$ converges uniformly for $\sigma$ in compact subsets of $(a,b)$, with $\varphi(\sigma+it)\to0$ as $|t|\to\infty$. Then the function $f$ defined by $f(x)=\frac{1}{2\pi i}\int_{c-i\infty}^{c+i\infty}\varphi(s)x^{-s}ds$ satisfies $\int_0^\infty x^{s-1}f(x) dx=\varphi(s)$ for $a<\re(s)<b$. Here $\varphi(s)=h_m(s)g(-s)$ is analytic on $0<\re(s)<\delta$, because the strip contains no integers, so that $\sin^m(\pi s)\ne0$ there, and $g(-s)$ is analytic because $\re(-s)>-\delta$. The required uniform integrability on every $[a,b]\subset(0,\delta)$ is \eqref{eq:vertical-integrand}. Therefore
\begin{equation}
\int_0^\infty x^{s-1}f_m(x) dx=h_m(s) g(-s)
=\frac{(-1)^{m-1}(m-1)! \pi^m}{\sin^m(\pi s)} g(-s),
\qquad 0<\re(s)<\delta,
\end{equation}
which is \eqref{eq:conjecture}. This proves Theorem~\ref{thm:A}.
\end{proof}

\begin{remark}\label{rem:continuation}
As in Hardy's theorem and in~\cite[Theorem~2]{bradshaw2025generalized}, the series \eqref{eq:fm-series} is guaranteed to converge only on $(0,e^{-P})$, and on the rest of $(0,\infty)$, the integrand of \eqref{eq:conjecture} is understood as the function $f_m$ defined by the inverse Mellin integral \eqref{eq:inverse-mellin}, which is its unique real-analytic continuation whenever one exists. In the applications of Section~\ref{sec:applications}, the continuation is exhibited in closed form.
\end{remark}

\begin{remark}\label{rem:weaker}
For $m\ge2$, the hypothesis $A<m\pi$ of Theorem~\ref{thm:A} is strictly weaker than Hardy's condition $A<\pi$. The $m$-fold decay $|\csc^m(\pi z)|\ll e^{-m\pi|\im(z)|}$ absorbs test functions $g$ growing almost as fast as $e^{m\pi|\im(z)|}$ on vertical lines. For example, $g(z)=1/\Gamma(z+1)^{2m-1}$ is admissible in Theorem~\ref{thm:A} for the given $m$ but fails Hardy's condition for every $m\ge2$. Indeed, the argument establishing \eqref{eq:invgamma-bound} in the proof of Corollary~\ref{cor:gamma} below in fact gives $|1/\Gamma(z+1)|\le C_{\delta,\varepsilon} e^{(\pi/2+\varepsilon)|\im(z)|}$ on $H(\delta)$ for every $\varepsilon>0$, so this $g$ satisfies \eqref{eq:growth-A} with any $A>(2m-1)\tfrac{\pi}{2}$, and $(2m-1)\tfrac{\pi}{2}<m\pi$ while $(2m-1)\tfrac{\pi}{2}>\pi$ for $m\ge2$.
\end{remark}

\section{Special cases and applications}\label{sec:applications}

The case $g=1$ of Theorem~\ref{thm:A} admits a complete independent verification, which we present first. Recall the classical Mellin evaluations, valid for $0<\re(s)<1$,
\begin{equation}\label{eq:classical-mellin}
\int_0^\infty\frac{x^{s-1}}{1+x} dx=\frac{\pi}{\sin(\pi s)},
\end{equation}
and
\begin{equation}
\mathrm{P.V.}\!\int_0^\infty\frac{x^{s-1}}{1-x} dx=\pi\cot(\pi s).
\end{equation}

\begin{proposition}[Differential identities]\label{prop:differential}
Let $s\in\mathbb{C}\setminus\mathbb{Z}$. Then for every odd $m\ge1$,
\begin{equation}\label{eq:diff-identity-odd}
P_m\Bigl(\frac{d}{ds}\Bigr)\Bigl[\frac{\pi}{\sin(\pi s)}\Bigr]=\frac{(-1)^{m-1}(m-1)! \pi^m}{\sin^m(\pi s)},
\end{equation}
and for every even $m\ge2$,
\begin{equation}\label{eq:diff-identity-even}
P_m\Bigl(\frac{d}{ds}\Bigr)\bigl[\pi\cot(\pi s)\bigr]=\frac{(-1)^{m-1}(m-1)! \pi^m}{\sin^m(\pi s)} .
\end{equation}
\end{proposition}

\begin{proof}
Write $ h_m(s)=(-1)^{m-1}(m-1)! \pi^m\csc^m(\pi s)$. By Lemma~\ref{lem:ode},
\begin{equation}\label{eq:Rm-rec}
\begin{aligned}
 h_m''+m^2\pi^2 h_m&=(-1)^{m-1}(m-1)!\bigl[C_m''+m^2\pi^2C_m\bigr]\\
&=(-1)^{m-1}(m-1)! m(m+1) C_{m+2}\\
&=(-1)^{m+1}(m+1)! C_{m+2}\\&= h_{m+2},
\end{aligned}
\end{equation}
that is, $\bigl(\frac{d^2}{ds^2}+m^2\pi^2\bigr) h_m= h_{m+2}$. For the odd chain, $P_1(\frac{d}{ds})[\pi\csc(\pi s)]=\pi\csc(\pi s)= h_1$; and if $P_m(\frac{d}{ds})[\pi\csc(\pi s)]= h_m$, then
\begin{equation}
P_{m+2}\Bigl(\frac{d}{ds}\Bigr)\bigl[\pi\csc(\pi s)\bigr]
=\Bigl(\frac{d^2}{ds^2}+m^2\pi^2\Bigr)P_m\Bigl(\frac{d}{ds}\Bigr)\bigl[\pi\csc(\pi s)\bigr]
=\Bigl(\frac{d^2}{ds^2}+m^2\pi^2\Bigr) h_m= h_{m+2},
\end{equation}
by \eqref{eq:airault-rec} and \eqref{eq:Rm-rec}. For the even chain, $\frac{d}{ds}\bigl[\pi\cot(\pi s)\bigr]=-\pi^2\csc^2(\pi s)= h_2$, so $P_2(\frac{d}{ds})[\pi\cot(\pi s)]= h_2$, and the identical induction applies.
\end{proof}

The identities \eqref{eq:diff-identity-odd}--\eqref{eq:diff-identity-even} sit opposite the derivative-polynomial formalism of Hoffman~\cite{hoffman1995derivative} and Cvijovi\'c~\cite{cvijovic2009derivative}, in which each derivative $d^{n}/ds^{n}$ of $\csc$ or $\cot$ is expanded as a polynomial in $\cot$; Proposition~\ref{prop:differential} runs the other way, exhibiting the constant-coefficient operator $P_{m}(d/ds)$ that collapses such an expansion to the single power $\csc^{m}$. For even $m$, this is visible directly: $\csc^{m}=(1+\cot^{2})^{m/2}$, so \eqref{eq:diff-identity-even} is an identity between polynomials in $\cot(\pi s)$ once the left side is expanded.

\begin{corollary}[The case $g=1$]\label{cor:g1}
For every odd $m\ge1$ and $0<\re(s)<1$,
\begin{equation}\label{eq:g1-odd}
\int_0^\infty x^{s-1} \frac{P_m(\log(x))}{1+x} dx
=\frac{(m-1)! \pi^m}{\sin^m(\pi s)} .
\end{equation}
For every even $m\ge2$ and $0<\re(s)<1$, with the singularity of the integrand at $x=1$ removable,
\begin{equation}\label{eq:g1-even}
\int_0^\infty x^{s-1} \frac{P_m(\log(x))}{1-x} dx
=\frac{-(m-1)! \pi^m}{\sin^m(\pi s)} .
\end{equation}
In particular, $m=3$ gives $\displaystyle\int_0^\infty x^{s-1} \frac{\log^2(x)+\pi^2}{1+x} dx=\frac{2\pi^3}{\sin^3(\pi s)}$ and $m=2$ gives $\displaystyle\int_0^\infty x^{s-1} \frac{\log(x)}{1-x} dx=\frac{-\pi^2}{\sin^2(\pi s)}$.
\end{corollary}

\begin{proof}
Take $g=1$ in Theorem~\ref{thm:A}; then \eqref{eq:growth-A} holds with $C=1$, $P=A=0$, and any $0<\delta<1$, so the theorem applies on $0<\re(s)<\delta$ for every $\delta<1$, hence on $0<\re(s)<1$. Since $Dg=0$, we have $[P_m(D+L)g]_{z=n}=P_m(\log(x))$ for every $n$, and for $0<x<1$,
\begin{equation}
f_m(x)=P_m(\log(x))\sum_{n\ge0}\bigl((-1)^mx\bigr)^n=\frac{P_m(\log(x))}{1-(-1)^mx}.
\end{equation}
It remains to identify the continuation of $f_m$ to $(0,\infty)$, in the sense of Remark~\ref{rem:continuation}, with the displayed closed form. The inverse Mellin integral \eqref{eq:inverse-mellin} defines a real-analytic function of $x$ on $(0,\infty)$; indeed, it is analytic in the sector $|\arg(x)|<m\pi-A$, by \eqref{eq:vertical-integrand} and Morera's theorem. For odd $m$, the closed form $P_m(\log(x))/(1+x)$ is analytic on a complex neighborhood of $(0,\infty)$; the two agree on $(0,1)$, hence on $(0,\infty)$ by the identity theorem. For even $m$, Proposition~\ref{prop:airault}(iii) gives $P_m(0)=0$ with a simple zero of $P_m$ at $0$. Since $\log(x)$ has a simple zero at $x=1$, the numerator $P_m(\log(x))$ has a simple zero cancelling the simple zero of the denominator $1-x$, so $P_m(\log(x))/(1-x)$ is again analytic on a neighborhood of $(0,\infty)$, with value $-p_{m,1}$ at $x=1$, and the identity-theorem argument applies verbatim. The integrals converge absolutely at both ends, the integrand being $O(x^{\re(s)-1}|\log(x)|^{m-1})$ as $x\to0^+$ and $O(x^{\re(s)-2}\log^{m-1}(x))$ as $x\to\infty$.

Alternatively, and independently of Theorem~\ref{thm:A}, for odd $m$ one may insert $\log^j(x)$ under the first integral of \eqref{eq:classical-mellin}, which amounts to applying $\frac{d^j}{ds^j}$; differentiation under the integral sign is justified by the dominating bound $(x^{a-1}+x^{b-1})(1+|\log(x)|)^{j}/(1+x)\in L^1$ for $\re(s)\in[a,b]\subset(0,1)$. The left side of \eqref{eq:g1-odd} then equals $P_m(\frac{d}{ds})[\pi\csc(\pi s)]$, and Proposition~\ref{prop:differential} gives the result. For even $m$, an analogous verification through the principal-value evaluation in \eqref{eq:classical-mellin} is possible, but it requires justifying differentiation of a principal-value integral with respect to the parameter; since the first proof covers both parities, we do not pursue it.
\end{proof}

\begin{corollary}[The case $g(z)=a^z$]\label{cor:az}
Let $a>0$. For odd $m\ge1$ and $0<\re(s)<1$,
\begin{equation}
\int_0^\infty x^{s-1} \frac{P_m(\log(ax))}{1+ax} dx=\frac{(m-1)! \pi^m}{\sin^m(\pi s)} a^{-s},
\end{equation}
and the analogous formula with $1-ax$ and the sign $-1$ holds for even $m$.
\end{corollary}

\begin{proof}
The function $g(z)=a^z$ satisfies \eqref{eq:growth-A} with $A=0$ and $P=\log(a)$. Since $Da^z=\log(a) a^z$, one has $[(D+L)^ja^z]_{z=n}=(\log(a)+\log(x))^ja^n$, hence $[P_m(D+L)g]_{z=n}=P_m(\log(ax)) a^n$, and for $0<x<a^{-1}$ the series sums to $P_m(\log(ax))/(1-(-1)^max)$. The continuation and convergence arguments of Corollary~\ref{cor:g1} apply verbatim, and $g(-s)=a^{-s}$. Equivalently, one may substitute $x\mapsto x/a$ in Corollary~\ref{cor:g1}; the agreement of the two routes is a consistency check on Theorem~\ref{thm:A}.
\end{proof}

\begin{corollary}[The case $m=2$, $g(z)=1/\Gamma(z+1)$]\label{cor:gamma}
For $0<\re(s)<1$,
\begin{equation}\label{eq:BA-524}
\sum_{n=0}^{\infty}\Bigl[P_2\Bigl(\tfrac{d}{dz}+\log(x)\Bigr)\frac{1}{\Gamma(z+1)}\Bigr]_{z=n}x^n=-e^x \Gamma(0,x),
\qquad x>0,
\end{equation}
and consequently
\begin{equation}\label{eq:BA-525}
\int_0^\infty x^{s-1}e^x \Gamma(0,x) dx=\frac{\pi^2}{\sin^2(\pi s) \Gamma(1-s)}=\frac{\pi}{\sin(\pi s)} \Gamma(s).
\end{equation}
These are equations (5.24) to (5.26) of~\cite{bradshaw2025generalized}, obtained there by symbolic evaluation and stated as consequences of the conjecture; Theorem~\ref{thm:A} promotes them to theorems.
\end{corollary}

\begin{proof}
First the growth hypothesis. We claim that for each $\delta\in(0,1)$ there is $C_\delta$ with
\begin{equation}\label{eq:invgamma-bound}
\Bigl|\frac{1}{\Gamma(z+1)}\Bigr|\le C_\delta e^{2|\im(z)|},\qquad z\in H(\delta).
\end{equation}
Write $z=v+iw$ with $v\ge-\delta$, so that $\re(z+1)\ge1-\delta>0$. The Weierstrass product yields the classical modulus identity $|\Gamma(\zeta)|^{2}=\Gamma(\re(\zeta))^{2}\prod_{k\ge0}\bigl(1+\im(\zeta)^{2}/(\re(\zeta)+k)^{2}\bigr)^{-1}$ for $\re(\zeta)>0$, so that
\begin{equation}
\Bigl|\frac{1}{\Gamma(z+1)}\Bigr|=\frac{1}{\Gamma(v+1)}\prod_{k\ge0}\Bigl(1+\frac{w^{2}}{(v+1+k)^{2}}\Bigr)^{1/2}.
\end{equation}
The factor $k=0$ is at most $1+|w|/(1-\delta)$, and for $k\ge1$ we have $v+1+k\ge k$, so the remaining product is at most
\begin{equation}
\prod_{j\ge1}\Bigl(1+\frac{w^{2}}{j^{2}}\Bigr)^{1/2}=\Bigl(\frac{\sinh(\pi|w|)}{\pi|w|}\Bigr)^{1/2}\le e^{\pi|w|/2},
\end{equation}
by Euler's product for the sine and the elementary inequality $\sinh(T)\le Te^{T}$. Writing $\gamma_\delta$ for the positive minimum of $\Gamma$ on $[1-\delta,\infty)$, we conclude $|1/\Gamma(z+1)|\le\gamma_\delta^{-1}\bigl(1+|w|/(1-\delta)\bigr)e^{\pi|w|/2}$, and \eqref{eq:invgamma-bound} follows since $\pi/2<2$. Thus \eqref{eq:growth-A} holds with $P=0$ and $A=2<2\pi=m\pi$, and Theorem~\ref{thm:A} applies with $m=2$ on $0<\re(s)<\delta$ for every $\delta<1$.

Next the summand. With $g(z)=1/\Gamma(z+1)$ we have $g'(z)=-\psi(z+1)/\Gamma(z+1)$, and $P_2(D+L)g=(D+L)g$, so
\begin{equation}
\bigl[P_2(D+L)g\bigr]_{z=n}=\frac{\log(x)-\psi(n+1)}{n!} .
\end{equation}
We claim
\begin{equation}\label{eq:psi-gen}
\sum_{n=0}^{\infty}\frac{\psi(n+1)}{n!} x^n=e^x\bigl(\log(x)+\Gamma(0,x)\bigr),\qquad x>0.
\end{equation}
Indeed $\psi(n+1)=-\gamma+H_n$ with $H_n=\sum_{k=1}^n\frac1k=\int_0^1\frac{1-t^n}{1-t} dt$, so, interchanging sum and integral by monotone convergence, all terms being non-negative for $t\in(0,1)$,
\begin{equation}
\sum_{n\ge0}\frac{H_n}{n!}x^n=\int_0^1\frac{e^x-e^{xt}}{1-t} dt
=e^x\int_0^1\frac{1-e^{-x(1-t)}}{1-t} dt
=e^x\int_0^x\frac{1-e^{-u}}{u} du,
\end{equation}
by the substitution $u=x(1-t)$. Termwise integration of $\frac{1-e^{-u}}{u}=\sum_{k\ge1}\frac{(-1)^{k-1}u^{k-1}}{k!}$ gives $\int_0^x\frac{1-e^{-u}}{u}du=\sum_{k\ge1}\frac{(-1)^{k-1}x^k}{k\cdot k!}=\gamma+\log(x)+\Gamma(0,x)$, the last equality being the classical expansion of the exponential integral $\Gamma(0,x)=\int_x^\infty e^{-u}\frac{du}{u}$ \cite[6.6.2]{DLMF}. Hence $\sum H_nx^n/n!=e^x(\gamma+\log(x)+\Gamma(0,x))$, and subtracting $\gamma e^x=\sum\gamma x^n/n!$ proves \eqref{eq:psi-gen}. Therefore
\begin{equation}
\sum_{n\ge0}\frac{\log(x)-\psi(n+1)}{n!}x^n=e^x\log(x)-e^x\bigl(\log(x)+\Gamma(0,x)\bigr)=-e^x\Gamma(0,x),
\end{equation}
which is \eqref{eq:BA-524}; here the series converges for all $x>0$ and equals the real-analytic function $-e^x\Gamma(0,x)$ on all of $(0,\infty)$, so no continuation is needed. Theorem~\ref{thm:A} now gives, with $g(-s)=1/\Gamma(1-s)$,
\begin{equation}
\int_0^\infty x^{s-1}\bigl(-e^x\Gamma(0,x)\bigr) dx=\frac{-\pi^2}{\sin^2(\pi s)}\cdot\frac{1}{\Gamma(1-s)},
\end{equation}
and Euler's reflection formula $\Gamma(s)\Gamma(1-s)=\pi/\sin(\pi s)$ converts the right side into the second form displayed in \eqref{eq:BA-525}. The convergence of the integral is absolute for $0<\re(s)<1$, since $e^x\Gamma(0,x)\sim-\log(x)$ as $x\to0^+$ and $e^x\Gamma(0,x)\sim1/x$ as $x\to\infty$.
\end{proof}

We close this section with the application announced in the introduction. Define the multiplicative convolution of suitable functions on $(0,\infty)$ by
\begin{equation}
(f\star g)(x)=\int_0^\infty f\Bigl(\frac{x}{t}\Bigr)g(t) \frac{dt}{t},
\end{equation}
under which Mellin transforms multiply. The Cauchy kernel $K_1(x)=(1+x)^{-1}$ has transform $\pi/\sin(\pi s)$ on the strip $0<\re(s)<1$, so the $m$-fold convolution power of $K_1$ has transform $\pi^m/\sin^m(\pi s)$, and Theorem~\ref{thm:A} identifies this power in closed form.

\begin{proposition}[Convolution powers of the Cauchy kernel]\label{prop:convolution}
Define $K_1(x)=(1+x)^{-1}$ and $K_m=K_{m-1}\star K_1$ for $m\ge2$. Then for every $m\ge1$ and every $x>0$,
\begin{equation}\label{eq:Km-closed}
K_m(x)=\frac{(-1)^{m-1}}{(m-1)!}\cdot\frac{P_m(\log(x))}{1-(-1)^mx},
\end{equation}
where for even $m$, the right side is interpreted by its limit $p_{m,1}/(m-1)!$ at $x=1$. In particular, $K_2(x)=\dfrac{\log(x)}{x-1}$ and $K_3(x)=\dfrac{\log^2(x)+\pi^2}{2(1+x)}$, and
\begin{equation}\label{eq:Km-mellin}
\int_0^\infty x^{s-1}K_m(x) dx=\Bigl(\frac{\pi}{\sin(\pi s)}\Bigr)^{m},\qquad 0<\re(s)<1 .
\end{equation}
\end{proposition}

\begin{proof}
Write $C_m^{\mathrm{cl}}(x)$ for the right side of \eqref{eq:Km-closed} and note first that $C_m^{\mathrm{cl}}$ is continuous on $(0,\infty)$. For odd $m$ this is clear. For even $m$, Proposition~\ref{prop:airault}(iii) gives $P_m(0)=0$ with $P_m'(0)=p_{m,1}$, so $P_m(\log(x))$ has a simple zero at $x=1$ cancelling the simple zero of $1-x$, with limiting value
$\lim_{x\to1}\frac{P_m(\log(x))}{1-x}=\lim_{x\to1}\frac{p_{m,1}\log(x)+O(\log^2(x))}{1-x}=-p_{m,1}$,
so that, for even $m$, $C_m^{\mathrm{cl}}(1)=\frac{(-1)^{m-1}}{(m-1)!}\cdot(-p_{m,1})=\frac{p_{m,1}}{(m-1)!}$; for example $K_2(1)=1$. Moreover, from continuity together with the asymptotics $P_m(\log(x))=O(|\log(x)|^{m-1})$ at $0$ and at $\infty$,
\begin{equation}\label{eq:Ccl-bound}
|C_m^{\mathrm{cl}}(y)|\le C_m \frac{(1+|\log(y)|)^{m-1}}{1+y},\qquad y>0 .
\end{equation}

\emph{Step 1, transforms.} We show by induction that $K_m$ is finite almost everywhere, non-negative, locally integrable, and
\begin{equation}\label{eq:Km-transform}
\int_0^\infty x^{s-1}K_m(x) dx=\Bigl(\frac{\pi}{\sin(\pi s)}\Bigr)^{m},\qquad 0<\re(s)<1,
\end{equation}
with absolute convergence. For $m=1$, this is the first evaluation in \eqref{eq:classical-mellin}, with absolute convergence clear. Assume it for $m$. Since $K_m\ge0$ and $K_1\ge0$, Tonelli's theorem applies to the non-negative double integral. For real $\sigma\in(0,1)$, substituting $x=ty$,
\begin{equation}
\begin{aligned}
\int_0^\infty x^{\sigma-1}K_{m+1}(x) dx
&=\int_0^\infty\!\!\int_0^\infty x^{\sigma-1}K_m\Bigl(\frac{x}{t}\Bigr)\frac{1}{1+t} \frac{dt}{t} dx\\
&=\int_0^\infty\frac{t^{\sigma-1}}{1+t} dt\int_0^\infty y^{\sigma-1}K_m(y) dy\\
&=\Bigl(\frac{\pi}{\sin(\pi\sigma)}\Bigr)^{m+1},
\end{aligned}
\end{equation}
which is finite; hence $K_{m+1}<\infty$ almost everywhere, $K_{m+1}\ge0$, and $x^{\sigma-1}K_{m+1}\in L^1$. For complex $s=\sigma+i\tau$ the same computation with Fubini's theorem gives \eqref{eq:Km-transform}, the required absolute convergence being the $\sigma$-case just established, since $|x^{s-1}|=x^{\sigma-1}$.

\emph{Step 2, uniqueness.} Fix $\sigma\in(0,1)$ and define $f(y)=e^{\sigma y}K_m(e^{y})$ and $g(y)=e^{\sigma y}C^{\mathrm{cl}}_m(e^{y})$ for $y\in\mathbb{R}$. By Step~1 and by \eqref{eq:Ccl-bound}, both belong to $L^1(\mathbb{R})$, and by \eqref{eq:Km-transform} and Corollary~\ref{cor:g1} their Fourier transforms coincide. For $\tau\in\mathbb{R}$,
\begin{equation}
\widehat f(\tau)=\int_{\mathbb{R}}e^{-i\tau y}e^{\sigma y}K_m(e^y) dy=\int_0^\infty x^{\sigma-i\tau-1}K_m(x) dx
=\Bigl(\frac{\pi}{\sin(\pi(\sigma-i\tau))}\Bigr)^{m}=\widehat g(\tau),
\end{equation}
the evaluation of $\widehat g$ being Corollary~\ref{cor:g1} multiplied by $\frac{(-1)^{m-1}}{(m-1)!}\cdot(-1)^{m-1}(m-1)!=1$. By the uniqueness theorem for Fourier transforms of $L^1$ functions~\cite{titchmarsh1948introduction}, $f=g$ almost everywhere, that is $K_m=C^{\mathrm{cl}}_m$ almost everywhere on $(0,\infty)$.

\emph{Step 3, everywhere.} We upgrade to pointwise equality by showing inductively that $K_m$ is continuous. $K_1$ is continuous and equals $C^{\mathrm{cl}}_1$. Suppose $K_m=C^{\mathrm{cl}}_m$ everywhere; in particular $K_m$ is continuous. For fixed $x>0$, we then have
\begin{equation}
K_{m+1}(x)=\int_0^\infty C^{\mathrm{cl}}_m\Bigl(\frac{x}{t}\Bigr)\frac{1}{(1+t)t} dt .
\end{equation}
The integrand is jointly continuous in $(x,t)\in(0,\infty)^2$, and for $x$ in a compact interval $[a,b]\subset(0,\infty)$ the bound \eqref{eq:Ccl-bound} gives the $x$-independent dominating function
\begin{equation}
C_m \frac{\bigl(1+|\log(a)|+|\log(b)|+|\log(t)|\bigr)^{m-1}}{\bigl(1+\frac{a}{t}\bigr)(1+t) t}\in L^1\bigl((0,\infty),dt\bigr),
\end{equation}
integrability at $t\to0^+$ coming from the factor $(1+\tfrac at)^{-1}t^{-1}\le a^{-1}$ and at $t\to\infty$ from $(1+t)^{-1}t^{-1}$. By dominated convergence, $K_{m+1}$ is continuous on $(0,\infty)$; being almost everywhere equal to the continuous function $C^{\mathrm{cl}}_{m+1}$ by Step~2, it equals it everywhere. This closes the induction, proves \eqref{eq:Km-closed}, and \eqref{eq:Km-mellin} is \eqref{eq:Km-transform}.
\end{proof}

\begin{remark}\label{rem:probabilistic}
The identity \eqref{eq:Km-closed} has a probabilistic reading which both explains its positivity content and locates it in an older literature. The substitution $x=e^{y}$ converts multiplicative convolution into ordinary convolution on $\mathbb{R}$, and
\begin{equation}\label{eq:sech}
  \frac{1}{\pi}\,e^{y/2}K_{1}(e^{y})
  \;=\;\frac{1}{2\pi}\operatorname{sech}\!\left(\frac{y}{2}\right)
\end{equation}
is a probability density with characteristic function $\operatorname{sech}(\pi\tau)$, a scaled hyperbolic secant law. Hence $\pi^{-m}e^{y/2}K_{m}(e^{y})$ is the density of a sum of $m$ independent copies, and \eqref{eq:Km-closed} is equivalent to the assertion that this density equals
\begin{equation}\label{eq:density}
  \frac{P_{m}(y)}{2\pi^{m}(m-1)!}
  \times
  \begin{cases}
    \operatorname{sech}(y/2), & m\text{ odd},\\[2pt]
    \operatorname{csch}(y/2), & m\text{ even}.
  \end{cases}
\end{equation}

In particular, since $K_{m}\geq 0$ by construction, \eqref{eq:Km-closed} records a positivity property of the polynomials $P_{m}$: the functions $(-1)^{m-1}P_{m}(\log(x))$ and $1-(-1)^{m}x$ have the same sign on $(0,\infty)$. For odd $m$ this is the trivial positivity of $P_{m}(t)=\prod(t^{2}+j^{2}\pi^{2})$, and for even $m$ it is the statement that $P_{m}(t)=t\prod(t^{2}+j^{2}\pi^{2})$ changes sign with $t=\log(x)$ exactly as $x-1$ does. The reading \eqref{eq:density} makes the property automatic rather than incidental, the two parities being the two cases of the single fact that a density is nonnegative.

The density of such a sum was first obtained by Baten~\cite{baten1934probability}, and the associated family of laws was developed by Harkness and Harkness~\cite{harkness1968generalized}. The closely related transforms $\int_{\mathbb{R}}(y/\sinh y)^{\ell}e^{ixy}\,dy$, together with a differential-difference recursion equivalent to Lemma~\ref{lem:ode}, appear in Pitman and Yor~\cite{pitman2003infinitely}, and the polynomials $P_{m}$ themselves in Airault's work~\cite{airault2008hyperbolic}. What \eqref{eq:Km-closed} adds is the factored closed form on the multiplicative side, a single polynomial evaluated at $\log(x)$ over the elementary kernel $1-(-1)^{m}x$, with the zeros of the numerator located exactly by Proposition~\ref{prop:airault}(iii). We have not found this formulation in the literature.
\end{remark}

\section{Concluding remarks}

Conjecture~1 of~\cite{bradshaw2025generalized} holds under the growth condition $A<m\pi$, strictly weaker than Hardy's for $m\geq 2$, and the mechanism of the proof is worth isolating. A second-order differential recursion in the pole-order parameter, Lemma~\ref{lem:ode}, translates through the residue calculus into a two-term recursion on principal-part coefficients, and the polynomial family solving that recursion is the family $P_{m}$ of \eqref{eq:airault-rec}. The same mechanism seems likely to apply whenever the principal parts of a family of kernels are generated from a fixed member by a linear differential equation in the order parameter, and it would be interesting to identify further kernels of this type among those admissible in~\cite[Theorem~5]{bradshaw2025generalized}.

The threshold $m\pi$ is optimal. Take $g(z)=\sin^{m}(\pi z)$, which is entire and satisfies \eqref{eq:growth-A} with $C=1$, $P=0$ and $A=m\pi$, since $|\sin(\pi z)|\leq \cosh(\pi w)\leq e^{\pi|w|}$. This $g$ vanishes to order $m$ at every integer, so $g^{(k)}(n)=0$ for all $n\geq 0$ and all $k\leq m-1$; as $\deg P_{m}=m-1$, every term of \eqref{eq:fm-series} vanishes and the left side of \eqref{eq:conjecture} is identically zero. The right side, however, is the nonzero constant
\begin{equation}
  (-1)^{m-1}(m-1)!\,\frac{\pi^{m}}{\sin^{m}(\pi s)}\,\sin^{m}(-\pi s)
  \;=\;-(m-1)!\,\pi^{m}.
\end{equation}
Hence $A<m\pi$ cannot be relaxed to $A\leq m\pi$, and $m\pi$ cannot be replaced by any larger constant. For $m=1$ this is the classical obstruction $g(z)=\sin(\pi z)$ to Hardy's theorem, and the general case shows that the admissible exponential type scales exactly with the pole order.

\bibliographystyle{unsrt}
\bibliography{refs}

\end{document}